\documentclass[a4paper,10pt]{article}

\usepackage{amsfonts,amssymb,amsmath}

\numberwithin{equation}{section}

\usepackage[normalem]{ulem}
\usepackage{xypic}
\usepackage{xcolor}
\usepackage{tikz-cd}
\newtheorem{theo}{Theorem}[section]
\newtheorem{coro}[theo]{Corollary}
\newtheorem{lemm}[theo]{Lemma}
\newtheorem{prop}[theo]{Proposition}
\newtheorem{defi}[theo]{Definition}
\newtheorem{rema}[theo]{Remark}

\newenvironment{proof}{\noindent \textbf{{Proof.}} \sf}

\def\qed{\hfill $\diamond$ \bigskip}

\def\lim{\mathop{\rm lim}\nolimits}

\def\Tor{\mathsf{Tor}}

\def\Ker{\mathsf{Ker}}
\def\ker{\mathop{\rm Ker}\nolimits}
\def\Coker{\mathsf{Coker}}

\def\Im{\mathsf{Im}}

\begin{document}
\sf

\title{Jacobi-Zariski long nearly exact sequences for associative algebras}
\author{Claude Cibils,  Marcelo Lanzilotta, Eduardo N. Marcos,\\and Andrea Solotar
\thanks{\footnotesize This work has been supported by the projects  UBACYT 20020130100533BA, PIP-CONICET 11220150100483CO, USP-COFECUB.
The third mentioned author was supported by the thematic project of FAPESP 2018/23690-6, and acknowledges support from the ``Brazilian-French Network in Mathematics" and from the CNPq research grant 302003/2018-5. The fourth mentioned author is a research member of CONICET (Argentina) and a Senior Associate at ICTP.}}

\date{}
\maketitle
\begin{abstract}
For an extension of associative algebras $B\subset A$ over a field and an $A$-bimodule $X$, we obtain a Jacobi-Zariski long nearly exact sequence relating the Hochschild homologies of $A$ and $B$, and the relative Hochschild homology, all of them with coefficients in $X$. This long sequence is exact twice in three. There is a spectral sequence which converges to the gap of exactness.
\end{abstract}



\section{\sf Introduction}

Let $k$ be a field. A long sequence of vector spaces and linear maps is called \emph{nearly exact} if it is a complex which is exact twice in three, see Definition \ref{definition long nearly exact}. Its homology at the spots where it is not exact is a graded vector space called the gap of the sequence.

Let $B\subset A$ be an extension of associative $k$-algebras. Let $X$ be an $A$-bimodule. In this paper we show that there is a Jacobi-Zariski long nearly exact sequence relating the Hochschild homologies of $A$ and $B$, and the relative one, all of them with coefficients in $X$. Moreover, there is a spectral sequence converging to the gap of this Jacobi-Zariski nearly exact sequence.  If $\Tor_*^B(A/B, (A/B)^{\otimes_B n})=0$ for $*>0$ and for all $n$, then its first page is given by $E^1_{p,q}=\Tor_{q}^{B^e}(X, (A/B)^{\otimes_B p})$ for $p,q>0$ and zeros elsewhere.

If $A/B$ is flat as a $B$-bimodule, A. Kaygun in \cite{KAYGUN,KAYGUNe} has obtained a Jacobi-Zariski long exact sequence. The long nearly exact sequence that we work out in this paper specialises to Kaygun's long exact sequence. Indeed, if $A/B$  is flat then the spectral sequence converges to $0$ and the nearly exact sequence is actually exact. On the other hand, {we improve} the Jacobi-Zariski long exact sequence of \cite{CLMS Pacific}. In this paper we do not assume that the bounded extension $B\subset A$ splits.

The Jacobi-Zariski sequence - also called transitivity sequence -  is originated in commutative algebra, see for instance \cite[p. 61]{ANDRE} or \cite{Iyengar}. Given a sequence of two morphisms between three commutative rings, there is an associated long exact sequence relating the Andr\'{e}-Quillen homology groups, see also \cite{GINOT}.  In characteristic zero, Andr\'{e}-Quillen homology is a direct summand of Hochschild homology defined in \cite{HOCHSCHILD1945}, see \cite{QUILLEN1968} and \cite{BARR}. Moreover, under flatness hypotheses,
 it is isomorphic to Harrison
homology with degree shifted by  1. An example of use of the Jacobi-Zariski exact sequence in commutative algebra can be found in \cite{PLANAS}. Concerning the origin of the name of Jacobi-Zariski given to the exact sequence for commutative rings, see \cite[p. 102]{ANDRE}.

In  \cite{CLMS Pacific,CLMS Pacific Erratum}, the Jacobi-Zariski sequence in non commutative algebra has been useful in relation to Han's conjecture (see \cite{HAN}). Moreover, it enables to give formulas for the change of dimension of Hochschild (co)homology when deleting or adding an inert arrow to the quiver of an algebra in \cite{CLMSS,CIBILSLANZILOTTAMARCOSSOLOTAR}.

The aim of this work is to develop the theory of the Jacobi-Zariski long nearly exact sequences for arbitrary extensions of algebras. In a forthcoming work, we will use it in relation to  Han's conjecture for a bounded extension of finite dimensional algebras which is non necessarily split. Moreover, the Jacobi-Zariski long nearly exact sequence will be also useful for describing the change of dimension of Hochschild (co)homology when adding or deleting non inert arrows.

Below we summarise the contents of this paper.

In Section \ref{relative bar normalised} we provide a brief account of relative Hochschild homology as defined by G. Hochschild in \cite{HOCHSCHILD1956}. Then we introduce the normalised relative bar resolution of an algebra $A$ with respect to a subalgebra $B$. Up to our knowledge, this resolution has been considered only once previously in \cite[p.56]{GERSTENHABERSCHACK1986}. If there exists a two-sided ideal $M$ such that $A=B\oplus M$, then the  normalised relative bar resolution is the one in \cite[Theorem 2.3]{CLMS Pacific}, see also \cite[Lemma 2.1]{CIBILS1990}. {When $B=k$ it} is the usual normalised bar resolution.

We provide a direct proof of the existence of the normalised relative bar resolution. It relies on the choice of a $k$-linear section $\sigma$ to the canonical projection $\pi:A\to A/B$. The proof  enables to set up tools and techniques for the rest of the paper.  Actually the differential does not depend on the choice of $\sigma$ but it is not provided by a simplicial module (see for instance \cite[p. 44]{LODAY1998}).

The normalised relative bar resolution enables us to provide a short nearly exact sequence that we call \emph{fundamental} in Section \ref{fundamental}. It is a short sequence of complexes which is exact except may be in its middle complex, where it has a ``gap complex''. For an $A$-bimodule $X$, this fundamental sequence is based on the complex obtained from the normalised relative bar resolution, as well as on the usual complexes of $A$ and $B$ for Hochschild homology with coefficients in $X$.

In Section \ref{long nearly} we first prove a general result: to an arbitrary short nearly exact sequence of complexes, we associate a long nearly exact sequence in homology.  Specialising this procedure to the fundamental sequence, we obtain the aimed  Jacobi-Zariski long nearly exact sequence.

Moreover, we show that the homology of the gap complex of an arbitrary short nearly exact sequence is precisely the gap of the associated long nearly exact sequence. As before, we apply this to the Jacobi-Zariski long nearly exact sequence that we have obtained.

Section \ref{gap JZ} is devoted to approximate the gap of the Jacobi-Zariski sequence. This is achieved by firstly describing the gap complex of the fundamental sequence. Then, by choosing a linear section $\sigma$ to the canonical projection $\pi:A\to A/B$, we provide a filtration of the gap complex. In the associated graded complex of this filtration, $S=\sigma(A/B)$ plays an important role when endowed with the transported structure of $B$-bimodule isomorphic to $A/B$.  If $\Tor_*^B(A/B, (A/B)^{\otimes_B n})=0$ for $*>0$ and for all $n$, then we show that the homology of the quotients of the filtration are the required $\Tor$ vector spaces.

In the last section, as mentioned, we retrieve  previous results  from \cite{KAYGUN,KAYGUNe} and we improve the Jacobi-Zariski sequence from \cite{CLMS Pacific}.

\section{\sf Normalised relative bar resolution} \label{relative bar normalised}

Let $k$ be a field and let $B\subset A$ be an extension of $k$-algebras. The category of $A$-modules is exact with respect to short exact sequences of $A$-modules which are split as sequences of $B$-modules, see \cite{QUILLEN,BUHLER}. The relative projective $A$-modules are described for instance in \cite{CLMS Pacific}, they are $A$-direct summands of induced modules from $B$. A relative projective resolution of an $A$-module is made with relative projectives and has a $B$-contracting homotopy, see \cite[p. 250]{HOCHSCHILD1956}. This way for any right $A$-module $X$ and left $A$-module $Y$,  the vector spaces $\Tor_*^{A|B}(X,Y)$ are well defined.

The extension of algebras $B\subset A$ provides an extension of the enveloping algebras $B^e\subset A^e$. Let $X$ be a left $A^e$-module, that is  an $A$-bimodule. G. Hochschild  defined in \cite{HOCHSCHILD1956}
$$H_*(A|B,X)= \Tor_*^{A^e|B^e}(X,A)$$
as  the relative Hochschild homology of $X$.

As pointed out in \cite{CLMS Pacific}, it comes down to the same to consider the extensions $B\otimes A^{op}\subset A^e$ or $B^e\subset A^e$ for the above definitions.

We will now introduce the normalised relative bar resolution. Its existence in \cite{GERSTENHABERSCHACK1986} is based on the use of the reference \cite{MACLANE}, which relevance is not clear for us.

\begin{rema}
In general $A/B$ has no associative multiplicative structure. To get around this, we consider a $k$-section $\sigma$ to $\pi: A\to A/B$ the canonical $B$-bimodule projection, that is $\pi\sigma = 1$. Observe that in general $\sigma$ cannot be chosen to be a $B$-bimodule map. Through $\sigma$ there is a non associative product in $A/B$: if $\alpha, \alpha'\in A/B$, we consider $\pi(\sigma(\alpha)\sigma(\alpha'))\in A/B$. We will use this to define the differentials below.
\end{rema}

\begin{rema}
Each summand of the differential in the next proposition is not well defined with respect to tensor products over $B$. Nevertheless, their alternate sum is well defined. In other words the next resolution is not given by a simplicial module as defined for instance in \cite[p. 44]{LODAY1998}.
\end{rema}

\begin{prop} Let $B\subset A$ be an extension of algebras. There is a relative resolution of $A$ which we call the \emph{normalised} relative bar resolution:
\begin{equation*}\label{nbrr}\cdots\stackrel{d}{\to}
A\otimes_B(A/B)^{\otimes_Bm}\otimes_BA
\stackrel{d}{\to}\cdots\stackrel{d}{\to}A\otimes_BA/B\otimes_BA\stackrel{d}{\to}A\otimes_BA
\stackrel{d}{\to}A\to 0\end{equation*}
where the last $d$ is the product of $A$ and
\begin{align*}
d(a_0\otimes\alpha_{1}\otimes\dots\otimes&\alpha_{n-1}\otimes a_n)=  a_0\sigma(\alpha_1)\otimes\alpha_2\otimes\dots\otimes\alpha_{n-1}\otimes a_n+\\
&\sum_{i=1}^{n-2}(-1)^ia_0\otimes\alpha_{1}\otimes\dots\otimes \pi(\sigma(\alpha_i)\sigma(\alpha_{i+1}))\otimes\dots\otimes\alpha_{n-1}\otimes a_n +\\
&(-1)^{n-1}a_0\otimes\alpha_{1}\otimes\dots\otimes\sigma(\alpha_{n-1})a_n.
\end{align*}
The differential $d$ does not depend on the section $\sigma$.
\end{prop}
\begin{proof}
We outline the main steps and tools of the proof.
\begin{itemize}
\item The bimodules involved are induced bimodules, hence they are relative projective, see for instance \cite{CLMS Pacific}.
\item A somehow intricate but straightforward computation shows that $d$ is well defined with respect to tensor products over $B$. It uses repeatedly  the following key fact: if $b\in B$ and $\alpha \in A/B$, then there exists $c\in B$ such that $\sigma(b\alpha)= b\sigma(\alpha) + c$. Indeed, $\pi$ is a morphism of $B$-bimodules, hence $\pi(\sigma(b\alpha) - b\sigma(\alpha))=0$. The right analogous is also needed, namely there exists $c'\in B$ such that $\sigma(\alpha b)= \sigma(\alpha)b + c'$. The other important point to use is that the tensor products are over $B$ at the codomain.

    \item The maps $d$ are clearly $A$-bimodule morphisms.

    \item There is a contracting homotopy $s$ given by
    $$s(a_0\otimes\alpha_{1}\otimes\dots\otimes\alpha_{n-1}\otimes a_n)= 1\otimes \pi(a_0)\otimes\alpha_{1}\otimes\dots\otimes\alpha_{n-1}\otimes a_n.$$
    Note first that $s$ is well defined with respect to tensor products over $B$. Next, to check that $sd+ds=1$,  use that if $a\in A$ then there exists $c\in B$ such that $\sigma\pi(a)=a+c$.

    This contracting homotopy is a $B-A$-bimodule morphism.

    \item  We indicate two ways to verify that $d^2=0$.

    The first one relies on the fact that in each degree the codomain of $s$ is generated by $\Im s$ as an $A$-bimodule. In the inductive step of the proof one shows that $d^2s=0$, hence $d^2=0$.

    The second one is by doing the computation of $d^2$, using repeatedly that if $\alpha$ and $\alpha'$ are in $A/B$, then there exists $c\in B$ such that
    $$\sigma\left(\pi(\sigma(\alpha) \sigma(\alpha'))\right)= \sigma(\alpha)\sigma(\alpha')+c.$$

    \item Finally let $\sigma'$ be another $k$-section. If $\alpha\in A/B$ then there exists $c_\alpha\in B$ such that $\sigma'(\alpha)=\sigma(\alpha)+c_\alpha$. Using this, an elaborate but not difficult computation shows that $d$ does not depend on the section.
\qed
    \end{itemize}
\end{proof}

\begin{coro}
  Let $B\subset A$ be an extension of $k$-algebras, let $\sigma$ be a $k$-linear section of $\pi:A\to A/B$, and let $X$ be an $A$-bimodule.

Then $H_*(A|B,X)$ is the homology of the chain complex $C_*(A|B,X)$:
  \begin{equation}\ \ \cdots  \stackrel{b_{A|B}}{\to} X\otimes_{B^e} (A/B)^{\otimes_Bm}\stackrel{b_{A|B}}{\to}\cdots \stackrel{b_{A|B}}{\to} X\otimes_{B^e}A/B \stackrel{b_{A|B}}{\to} X_B\to 0\end{equation}
  where $X_B= X\otimes_{B^e}B=X/\langle bx-xb\rangle = H_0(B,X)$ and
  \begin{align*}
b_{A|B}(x\otimes\alpha_{1}\otimes\dots\otimes&\alpha_{n})=  x\sigma(\alpha_1)\otimes\alpha_2\otimes\dots\otimes\alpha_{n}+\\
&\sum_{i=1}^{n-1}(-1)^i \ x\otimes\alpha_{1}\otimes\dots\otimes \pi(\sigma(\alpha_i)\sigma(\alpha_{i+1}))\otimes\dots\otimes\alpha_{n}  +\\
&(-1)^{n}\ \sigma(\alpha_n)x\otimes\alpha_{1}\otimes\dots\otimes\alpha_{n-1}.
\end{align*}

Moreover the differential $b_{A|B}$ does not depend on the choice of the linear section $\sigma$.

\end{coro}
\section{\sf Fundamental short nearly exact sequence for relative homology}\label{fundamental}
In this section we associate to an extension of algebras a fundamental short nearly exact sequence of complexes.
 \begin{defi}\label{sequence nearly exact}
Let \begin{equation}\label{seq}
      0\to C_*\stackrel{\iota}{\to}  D_*\stackrel{\kappa}{\to} E_*\to 0
      \end{equation}  be a sequence of positively graded chain complexes of vector spaces.
It is called  \emph{short nearly exact } if $\iota$ is injective, $\kappa$ is surjective and $\kappa\iota=0$.

The  complex $\left({\Ker\kappa}/{\Im \iota}\right)_*$ is its \emph{gap complex}.
 \end{defi}

 \begin{lemm}\label{snexact by columns}
Let (\ref{seq}) be a short nearly exact sequence of {positively graded} complexes, and consider it as a double complex after the standard change of signs. Then the spectral sequence obtained by filtering the double complex by rows converges to the homology of the gap complex.  \end{lemm}
\begin{proof}
At page $1$, the spectral sequence given by the filtration by rows has a single column at $p=1$ which is the gap complex, its boundaries are induced by those of  $D_*$. Hence in page $2$ we also have a single column at $p=1$, its values are the homology of the gap complex. The differentials come from $0$ or go to $0$, so these spaces live forever and the spectral sequence converges to them. \qed
\end{proof}

Let $A$ be a $k$-algebra and let $X$ be an $A$-bimodule. We denote
$$C_*(A,X): \cdots \stackrel{b_A}{\to}X\otimes A^{\otimes m}\stackrel{b_A}{\to} \cdots \stackrel{b_A}{\to}X\otimes A\otimes A \stackrel{b_A}{\to}X\otimes A\stackrel{b_A}{\to} X\to 0$$  the usual complex which computes the Hochschild homology $H_*(A,X)$.
To obtain a nearly exact sequence of complexes, we consider a slightly modified truncated complex in degrees $0$ and $1$, without changing degrees, as follows:
$$\check{C}_*(A,X): \cdots \stackrel{b_A}{\to}X\otimes A^{\otimes m}\stackrel{b_A}{\to} \cdots \stackrel{b_A}{\to}X\otimes A\otimes A \stackrel{b_A}{\to} \Ker b_A\to 0.$$
This is well defined since $\Im b_A \subset \Ker b_A$. This complex still computes $H_*(A,X)$ except in degree $0$. Similarly we consider
\begin{align*}
\check{C}_*(A|B,X):
  \cdots  \stackrel{b_{A|B}}{\to} X\otimes_{B^e} A/B^{\otimes_Bn}\stackrel{b_{A|B}}{\to}\cdots \stackrel{b_{A|B}}{\to} X\otimes_{B^e}A/B\otimes_B A/B \\\stackrel{b_{A|B}}{\to} \Ker b_{A|B}\to 0
  \end{align*}
  which still computes $H_*(A|B,X)$ in positive degrees.

\begin{theo}\label{HH_ seq nearly exact complexes}
Let $B\subset A$ be an extension of $k$-algebras and let $X$ be an $A$-bimodule. There is a short sequence of positively graded chain complexes
\begin{equation}\label{nearly exact chain}
0\to \check{C}_*(B,X)\stackrel{\iota}{\to}  \check{C}_*(A,X)\stackrel{\kappa}{\to} \check{C}_*(A|B,X)\to 0
\end{equation}
which is nearly exact, except in degree $1$ where $\kappa$ is not necessarily surjective. It will be called the \emph{fundamental sequence}.
\end{theo}
\begin{proof}
The map $\iota$ is clearly injective and is trivially a map of complexes. For $n\geq 1$ we set
$$\kappa(x\otimes a_1\otimes\dots\otimes a_n)= x\otimes \pi(a_1)\otimes\dots\otimes \pi(a_n)$$
which is surjective and verifies $\kappa\iota=0$. It remains to prove that $\kappa$ is a map of complexes.

In degree zero of the original complexes, let $\kappa : X\to X_B$  be given by $\kappa(x)= \overline{x}$.  Notice that we have $X\stackrel{\iota=1}{\to} X\stackrel{\kappa}{\to} X_B$ but $\kappa\iota\neq 0$. This is the reason for having considered the above truncated modification of the original complexes.

To obtain $$\kappa_| :\Ker b_A=\check{C}_1(A,X)\longrightarrow \check{C}_1(A|B,X)=\Ker b_{A|B}$$ we prove that the following diagram is commutative

\centerline
{\xymatrix{
X\otimes A \ar[d]^{b_A}          \ar[r]^-\kappa      & X\otimes_{B^e}A/B    \ar[d]^{b_{A|B}} \\
X \ar[r]^\kappa        & X_B
}}
Recall that $\sigma$ is a chosen linear section of the canonical projection $\pi: A\to A/B$. We have
\begin{itemize}
  \item $\kappa b_A(x\otimes a) = \kappa (xa - ax)= \overline{(xa - ax)}.$
  \item $b_{A|B}\kappa(x\otimes a)= b_{A|B}(x\otimes \pi(a))= \overline{x\sigma(\pi(a))} - \overline{\sigma(\pi(a))x}.$
\end{itemize}
There exists $c\in B$ such that $\sigma(\pi(a))= a+c$. Hence
$b_{A|B}\kappa(x\otimes a)= \overline{xa} + \overline{xc}- \overline{ax} -\overline{cx}.$ Finally observe that since $c\in B$, we have that $\overline{xc}=\overline{cx}$ as elements of $X_B$.

We infer the existence of the restriction $\kappa_| : \Ker b_A \to \Ker b_{A|B}. $ At degree $1$ of the original complexes we have $\kappa\iota=0$, hence the composition $$\Ker b_B\stackrel{\iota_|}{\to} \Ker b_A \stackrel{\kappa_|}{\to} \Ker b_{A|B}$$ is also zero. However $\kappa_|: \Ker b_A \to \Ker b_{A|B}$ is not surjective in general.

Next we verify that $\kappa$ is a morphism of complexes.

In degree $1$ we have
\begin{align*}\kappa b_A(x\otimes a_1 \otimes a_2)&= \kappa (xa_1 \otimes a_2 - x\otimes a_1a_2 + a_2x\otimes a_1)\\&=xa_1\otimes \pi(a_2) - x\otimes \pi(a_1a_2)+a_2x\otimes \pi(a_1).
\end{align*}
While
\begin{align*}
b_{A|B}\kappa (x\otimes &a_1 \otimes a_2)= b_{A|B}(x\otimes \pi(a_1)\otimes \pi(a_2))\\&= x\sigma(\pi(a_1))\otimes \pi(a_2) - x\otimes \pi(\sigma(\pi(a_1))\sigma(\pi(a_2)) + \sigma(\pi(a_2))x\otimes \pi(a_1).
\end{align*}

We have $\sigma(\pi(a_1))=a_1+c_1$ and $\sigma(\pi(a_2))=a_2+c_2$, for $c_1,c_2\in B.$ This way the last expression becomes:
\begin{align*}
xa_1\otimes \pi(a_2) + xc_1\otimes \pi(a_2) - x\otimes \pi(a_1a_2+a_1c_2+c_1a_2+c_1c_2) +\\a_2x\otimes \pi(a_1) +  c_2x\otimes \pi(a_1)\\
= xa_1\otimes \pi(a_2) + xc_1\otimes \pi(a_2) - x\otimes \pi(a_1a_2) - x\otimes \pi(a_1c_2) - x\otimes \pi(c_1a_2) +\\ a_2x\otimes \pi(a_1) + c_2x\otimes \pi(a_1)\\
= xa_1\otimes \pi(a_2) + xc_1\otimes \pi(a_2) - x\otimes \pi(a_1a_2) - x\otimes \pi(a_1)c_2 - x\otimes c_1\pi(a_2) +\\ a_2x\otimes \pi(a_1) +   c_2x\otimes \pi(a_1)\\
= xa_1\otimes \pi(a_2) + xc_1\otimes \pi(a_2) - x\otimes \pi(a_1a_2) - c_2x\otimes \pi(a_1) - xc_1\otimes \pi(a_2) +\\ a_2x\otimes \pi(a_1) +  c_2x\otimes \pi(a_1)\\
= xa_1\otimes \pi(a_2) - x\otimes \pi(a_1a_2) + a_2x\otimes \pi(a_1).
\end{align*}
Observe that we made use in an essential way that at the codomain of $\kappa$, the first tensor product is over $B^e$, and the other ones are over $B$. The proof in an arbitrary degree follows the same lines.
\qed
\end{proof}

\section{\sf Jacobi-Zariski long nearly exact sequence}\label{long nearly}

In this section we will obtain one of the main results of this work.

\begin{defi}\label{definition long nearly exact}
A complex of vector spaces ending at a fixed $n\geq 0$
\begin{equation*}
 \dots \stackrel{\delta}{\to} U_{m} \stackrel{I}{\to} V_{m} \stackrel{K}{\to} W_{m} \stackrel{\delta}{\to} U_{m-1} \stackrel{I}{\to} V_{m-1}\to \dots \stackrel{\delta}{\to} U_{n} \stackrel{I}{\to} V_{n} \stackrel{K}{\to} W_{n}
\end{equation*}
is a \emph{long nearly exact sequence} if it is exact except perhaps at $V_*$.

Its \emph{gap} is the graded vector space $(\Ker K/\Im I)_*$ .
\end{defi}
\begin{theo}\label{theo gap}
Let \begin{equation}\label{towards JZ nes}0\to C_*\stackrel{\iota}{\to}  D_*\stackrel{\kappa}{\to} E_*\to 0\end{equation} be a short nearly exact sequence of positively graded chain complexes (see Definition \ref{sequence nearly exact}).

Then there is a long nearly exact sequence
\begin{align*}\dots &\stackrel{\delta}{\to} H_{m}(C) \stackrel{I}{\to} H_{m}(D) \stackrel{K}{\to} H_{m}(E) \stackrel{\delta}{\to} H_{m-1}(C)  \stackrel{I}{\to} \dots \\&\stackrel{\delta}{\to} H_{0}(C) \stackrel{I}{\to} H_{0}(D) \stackrel{K}{\to} H_{0}(E)\to 0.
\end{align*}
Moreover its gap is isomorphic to $H_*(\Ker\kappa/\Im\iota)$, namely to the homology of the gap complex of (\ref{towards JZ nes}).

If (\ref{towards JZ nes}) is nearly exact except at the lowest degree where $\kappa$ is not surjective, then the same  holds except that  $H_{0}(D) \stackrel{K}{\to} H_{0}(E)$ is not surjective. In other words there is a long nearly exact sequence
\begin{align*}\dots &\stackrel{\delta}{\to} H_{m}(C) \stackrel{I}{\to} H_{m}(D) \stackrel{K}{\to} H_{m}(E) \stackrel{\delta}{\to} H_{m-1}(C)  \stackrel{I}{\to} \dots \\&\stackrel{\delta}{\to} H_{0}(C) \stackrel{I}{\to} H_{0}(D) \stackrel{K}{\to} H_{0}(E)
\end{align*}
with gap as before.

\end{theo}
\begin{rema}
The second part of the statement takes into account the specific situation occurring in degree $1$ of Theorem \ref{HH_ seq nearly exact complexes}.
\end{rema}

\begin{proof}
After the standard change of signs, we consider the short nearly exact sequence (\ref{towards JZ nes}) as a double complex with three columns at $p=0,1$ and $2$. By Lemma \ref{snexact by columns},  filtering by rows gives a spectral sequence converging to  $H_*(\Ker\kappa/\Im\iota)$. On columns $0$ and $2$, starting at page $1$, they are zeros. Indeed, $\iota$ is injective and $\kappa$ is surjective.

Consider the filtration by columns, let $I$ and $K$ be the horizontal maps at page $1$. At page $2$, first consider the columns $p=0$ and $p=2$ where we have respectively $\Coker K$ and $\Ker I$. The differential $d_2: \Ker I \to \Coker K$  lowers the total degree by $1$. We claim it is invertible. Indeed, at page $3$ the column $p=2$ consists of $\Ker d_2$, while at the same page, column $p=0$ is  $\Coker d_2$. Actually these vector spaces live forever since $d_3$ at these spots come from $0$ or go to $0$. By the analysis of the filtration by rows of the double complex, we infer that $\Ker d_2=0$ and $\Coker d_2=0$, that is $d_2: \Ker I \to \Coker K$ is an isomorphism.

The morphism $\delta$ is then obtained by pre-composing $d_2^{-1}$ with the canonical projection to $\Coker K$ and post-composing with the inclusion of $\Ker I$. By construction $\Ker\delta=\Im K$ and $\Im \delta= \Ker I$.

Still at  page $2$ but at column $p=1$, we have $\Ker K/\Im I$. At these spots $d_2$ comes from $0$ or goes to $0$. Therefore $\Ker K/\Im I$ lives forever.

We use again that both filtrations converge to the same graded vector space to infer that $H_*(\Ker\kappa/\Im\iota)$ is isomorphic to $(\Ker K/\Im I)_*$, that is to the gap of the long nearly exact sequence.

Finally, if (\ref{towards JZ nes}) is nearly exact except in the lowest degree where $\kappa$ is not surjective, then the filtration by columns at page $1$ gives in addition $\Coker \kappa$ at the spot $(0,0)$ which lives forever. At the second page $H_*(\Ker\kappa/\Im\iota)$ for $*\geq 1$ is not affected and lives forever. The proof that the lowest $d_2$ from $(2,0)$ to $(0,1)$ is invertible remains true.
\qed
\end{proof}
\begin{theo}\label{JZ}
Let $B\subset A$ be an extension of $k$-algebras, and let $X$ be an $A$-bimodule.

Then there is a Jacobi-Zariski long nearly exact sequence in Hochschild homology
\begin{align*} \dots \stackrel{\delta}{\to} H_{m}(B,X) \stackrel{I}{\to} H_{m}(A,X) \stackrel{K}{\to} H_{m}(A|B,X) \stackrel{\delta}{\to}H_{m-1}(B,X) \stackrel{I}{\to}\dots \\\stackrel{\delta}{\to} H_{1}(B,X) \stackrel{I}{\to} H_{1}(A,X) \stackrel{K}{\to} H_{1}(A|B,X).
\end{align*}
\end{theo}

\begin{proof}
Theorem \ref{HH_ seq nearly exact complexes} establishes that the fundamental sequence \begin{equation*}
0\to \check{C}_*(B,X)\stackrel{\iota}{\to}  \check{C}_*(A,X)\stackrel{\kappa}{\to} \check{C}_*(A|B,X)\to 0
\end{equation*}
is nearly exact except in its lowest degree where $\kappa$ is not surjective.
The second part of the previous result provides the proof of the statement.
\qed
\end{proof}

\section{\sf Gap of the Jacobi-Zariski long nearly exact sequence}\label{gap JZ}

Next we will approximate the gap of the Jacobi-Zariski long nearly exact sequence of Theorem \ref{JZ}. We assume that  $\Tor_*^B(A/B, (A/B)^{\otimes_B n})=0$ for $*>0$ and for all $n$. Note that this is fulfilled if $A/B$ is either a left or a right projective $B$-module.

\begin{theo}\label{homology of JZ}
Let $B\subset A$ be an extension of $k$-algebras, and let $X$ be an $A$-bimodule. Let
\begin{align*} \dots \stackrel{\delta}{\to} H_{m}(B,X) \stackrel{I}{\to} H_{m}(A,X) \stackrel{K}{\to} H_{m}(A|B,X) \stackrel{\delta}{\to}H_{m-1}(B,X) \stackrel{I}{\to}\dots \\\stackrel{\delta}{\to} H_{1}(B,X) \stackrel{I}{\to} H_{1}(A,X) \stackrel{K}{\to} H_{1}(A|B,X)
\end{align*}
be the Jacobi-Zariski long nearly exact sequence obtained in Theorem \ref{JZ}.

If $\Tor_*^B(A/B, (A/B)^{\otimes_B n})=0$ for $*>0$ and for all $n$,
then in degrees $\geq 2$ the gap $(\Ker K/\Im I)_*$ is approximated by a spectral sequence converging to it, whose terms at page $1$ are
$$E^1_{p,q}= \Tor^{B^e}_{q}(X,(A/B)^{\otimes_Bp}) \mbox{   \ \ for } p,q>0$$
and $0$ everywhere else.

If $A$ and $X$ are finite dimensional,  then  in degree $1$ we have that $K$ is surjective and $\Ker K= \Im I$.

\end{theo}
\begin{rema}
We underline that the degrees of the above $\Tor$ vector spaces at page $1$ are strictly positive.
\end{rema}

Before proving the theorem, we will describe the gap complex of  the fundamental sequence (\ref{nearly exact chain}).

Let $B\subset A$ be an inclusion of $k$-algebras, $\pi:A\to A/B$ the canonical $B$-bimodule map and $\sigma : A/B\to B$  a chosen $k$-section  to $\pi$. Let $S=\Im \sigma$. Then $A=B\oplus S$ as vector spaces. Of course $S$ and $A/B$ are isomorphic vector spaces via $\sigma$ and $\pi_{|_S}$.

\begin{rema} \label{transport} By transport of structure from $A/B$, the vector space $S$ can be endowed with a $B$-bimodule structure as follows. If $s\in S$ and $b\in B$, then
$$b.s = \sigma(b\pi(s)) \mbox{ and } s.b= \sigma(\pi(s)b).$$

Of course, the resulting $B$-bimodule $S$ is not a $B$-subbimodule of $A$ in general. However there exist $c$ and $c'\in B$ such that
$$b.s = bs+c  \mbox{ and } s.b= sb+c'$$
since $\pi(b.s - bs) = 0 = \pi(s.b - sb)$.
\end{rema}

\begin{defi} In the above situation $A=B\oplus S$, let $[S_pB_q]$ be the subspace of $A^{\otimes n}$ which is the direct sum of the direct summands of $A^{\otimes n}$ having $p$ tensorands in $S$ and $q$ tensorands in $B$, with $p+q=n$.
\end{defi}

For instance for $n=3$, $$[S_2B_1]= (S\otimes S\otimes B) \oplus (S\otimes B\otimes S) \oplus (B\otimes S\otimes S).$$

We have
$$A^{\otimes n}=\bigoplus_{\substack{p+q=n\\p\geq 0 \ q\geq 0}} [S_pB_q].$$

Recall that the map $\kappa: X\otimes A^{\otimes n} \to X\otimes_{B^e} (A/B)^{\otimes_B n}$ is given by $$\kappa(x\otimes a_1\otimes\dots\otimes a_n)= x\otimes \pi(a_1)\otimes\dots\otimes \pi(a_n).$$
\begin{defi}
For $n\geq 1$, let
\begin{equation}\label{definition subspace}
L_{n,0}=\Ker \kappa_{\big|_{X\otimes S^{\otimes n}}} : X\otimes S^{\otimes n}\twoheadrightarrow X\otimes_{B^e} (A/B)^{\otimes_B n}.
\end{equation}

\end{defi}
The index $0$ in $L_{n,0}$  underlines that there are no $B$-tensorands involved in its definition.
\begin{rema}\label{L}
Let $L_{n,0}'= \Ker  \left(X\otimes (A/B)^{\otimes n}\twoheadrightarrow X\otimes_{B^e} (A/B)^{\otimes_B n}\right) $,  that is $L_{n,0}'$ is the subspace which defines the tensor products over $B^e$ and $B$ in $X\otimes (A/B)^{\otimes n}$.

By transport of structure, we have endowed $S$ with a $B$-bimodule structure such that $\pi_{|_S}: S\to A/B$ is an isomorphism of $B$-bimodules. Note that $\pi_{|_S}(L_{n,0})=L_{n,0}'$. In other words, $$(X\otimes S^{\otimes n})/L_{n,0}= X\otimes_{B^e} S^{\otimes_B n}.$$
\end{rema}

The next result describes the gap complex of the fundamental sequence (\ref{nearly exact chain}).

\begin{lemm}\label{defect}
In the situation of Theorem \ref{HH_ seq nearly exact complexes}, let
\begin{equation*}
0\to \check{C}_*(B,X)\stackrel{\iota}{\to}  \check{C}_*(A,X)\stackrel{\kappa}{\to} \check{C}_*(A|B,X)\to 0
\end{equation*}
be the fundamental short nearly exact sequence of complexes (\ref{nearly exact chain}).

Then for $n\geq 2$ we have
$$
\arraycolsep=1mm\def\arraystretch{2}
\begin{array}{llll}
\Ker \kappa = L_{n,0}\oplus \left( \bigoplus_{\substack{p+q=n\\p\geq 0 \ q>0}}\ X\otimes [S_pB_q]\right) \\
\Im \iota = X\otimes B^{\otimes n} = X\otimes [S_0B_n]\\
(\Ker\kappa/\Im \iota)_n=  L_{n,0}\oplus \left(\bigoplus_{\substack{p+q=n\\p>0\ q>0}}\ X\otimes [S_pB_q]\right).
\end{array}
$$

\end{lemm}

\begin{proof}
For $n\geq 2$, consider the vector space decomposition $$X\otimes A^{\otimes n} = \bigoplus_{\substack{p+q=n\\p\geq 0\ q\geq 0}}\ X\otimes [S_pB_q].$$
If $q>0$, then $\kappa\left(X\otimes [S_pB_q]\right)=0$. Hence $$\bigoplus_{\substack{p+q=n\\p\geq 0\ q>0}}\ X\otimes [S_pB_q]\subset \Ker \kappa.$$
If $q=0$, then by definition $\Ker\kappa_{\big |_{X\otimes S^{\otimes n}}}=L_{n,0}$. It follows that $\Ker \kappa$ is as stated.
On the other hand, clearly $\Im \iota$ is the direct summand corresponding to $p=0$ and $q=n$; the result follows. \qed
\end{proof}

\begin{rema}
For $n=1$ there is a dead end as follows. According to Theorem \ref{HH_ seq nearly exact complexes} we should consider the sequence
\begin{equation}\label{ker}0\to \Ker b_B \stackrel{\iota_|}{\to} \Ker b_A \stackrel{\kappa_|}{\to} \Ker b_{A|B}\to 0\end{equation}
which is the well defined restriction of
\begin{equation}\label{one}0\to X\otimes B \stackrel{\iota}{\to} X\otimes A \stackrel{\kappa}{\to} X\otimes_{B^e} (A/B)\to 0.
\end{equation}
The maps $\iota$ and $\kappa$ are still injective and surjective respectively, and
$$\frac{\Ker\kappa}{\Im\iota} = \frac{(X\otimes B) \oplus L_{1,0}}{X\otimes B}= L_{1,0}.$$

In the sequence (\ref{ker}) we clearly have that $\iota_|$ is injective and $\kappa_|\iota_|=0$. However  $\Ker \kappa_| /\Im \iota_|$  is intricate to describe. Nevertheless, we will obtain from \cite{CLMSS} that if $A$ and $X$ are finite dimensional, the Jacobi-Zariski long nearly exact sequence is exact in degree $1$.

\end{rema}

\noindent\textbf{Proof of Theorem \ref{homology of JZ}.}
By Theorem \ref{theo gap} the gap of the Jacobi-Zariski long nearly exact sequence is the homology of the gap complex of the fundamental sequence. Therefore we focus on the latter.

By Lemma \ref{defect} the vector spaces of the gap complex  of the fundamental sequence are, for $n\geq 2$:
$$(\Ker \kappa/\Im \iota)_n=  L_{n,0}\oplus \left(\bigoplus_{\substack{p+q=n\\p>0\ q>0}}\ X\otimes [S_pB_q]\right).$$

We will show that there is a filtration $(G_p)_{p>0}$ of the gap complex such that $G_p/G_{p-1}$ (that is the column $p$ at page $0$ of the spectral sequence induced by the filtration) has the stated homology. In each degree the chains of $G_p$ have $p$ or less tensorands from $S$:
$$G_p= \bigoplus_{0<i\leq p} L_{i,0}\oplus \left(\bigoplus_{\substack{p \geq i>0\ q>0}}\ X\otimes [S_iB_q]\right).$$
Namely
$$\left(G_p\right)_n= (\Ker \kappa/\Im \iota)_n \mbox{ \ for }n\leq p, $$
and
$$\left(G_p\right)_n= \bigoplus_{\substack{i+q=n\\ p\geq i>0 \ q>0}}\ X\otimes [S_iB_q] \mbox{ \ for }n>p.$$

We recall that the {differentials} of the gap complex are induced from the {differentials} of the Hochschild complex $C_*(A,X)$.  The following observations show that $G_p$ is indeed a {subcomplex}:

\begin{enumerate}
\item If $s,s'\in S$, then there exists $s''\in S$ and $c\in B$ such that $ss' =s''+c$. In this case we have that the number of tensorands belonging to $S$ decreases by one when we apply the boundary formulas.
    \item \label{the exception} If $s\in S$ and $b\in B$ we have  $sb =c+s.b$, where  $c\in B$ and $s.b\in S$ is the right action of $B$ on $S$ by transport of structure, see Remark \ref{transport}.

        In the boundary formulas, the number of tensorands in $S$ decreases by one in the summand corresponding to $c$, and {it} is maintained in the other. Similarly for $bs$.
        \item If $x\in X$ and $s\in S$, then both $sx$ and $xs$ also have less tensorands in $S$.
    \end{enumerate}

For $p>0$, we have
$$(G_p/G_{p-1})_q= X\otimes[S_pB_q] \mbox{ \ \ if } q>0$$
and
$$(G_p/G_{p-1})_0= L_{p,0}.$$

Actually the three observations above show that the number of tensorands in $S$  decreases by one, except in case \ref{the exception}. In other words considering $S$ with its $B$-bimodule structure obtained by transport of structure, with  ``zero internal product" and with ``zero action on $X$" provides the same complex $G_p/G_{p-1}$.

This complex has been considered in the proof of \cite[Proposition 3.3]{CLMS Pacific,CLMS Pacific Erratum}, where we proved that  if  $\Tor_*^B(S, S^{\otimes_B n})=0$ for $*>0$ and for all $n$ then for $q>0$ we have $$H_q(G_p/G_{p-1})=\Tor^{B^e}_{q}(X,S^{\otimes_Bp})$$ while $H_0(G_p/G_{p-1})=0$.  This  finishes  the proof in degrees $\geq 2$ since the $B$-bimodules $S$ and $A/B$ are isomorphic, see Remark \ref{transport}.

For the convenience of the reader,  we next recall the proof that for $q>0$ we have $H_q(G_p/G_{p-1})=\Tor^{B^e}_{q}(X,S^{\otimes_Bp})$ while $H_0(G_p/G_{p-1})=0$.

The standing step is to replace $G_p/G_{p-1}$  by $G'_p$  in degree $0$ as follows: \\
 \centerline{$(G'_p)_0=X\otimes S^{\otimes p}=X\otimes [S_pB_0]$ while  $(G'_p)_q=  (G_p/G_{p-1})_q$ for $q>0$}. \\
 The boundaries of $G'_p$ are the same as those of $G_p/G_{p-1}$, this makes sense since $L_{p,0}\subset X\otimes S^{\otimes p}.$

For $p>0$ we assert that $H_q(G'_p)=\Tor_q^{B^e}(X,S^{\otimes_Bp})$ for all $q\geq 0$.

We consider below a projective resolution of a $B^e$-module $S$ given in \cite[Proposition 4.1]{CIBILS2000}.  Applying the functor $X\otimes_{B^e}-$ to it gives $G'_1$, which proves the assertion for $p=1$.  For tensor products of modules over $k$, we omit in what follows the tensor product sign $\otimes$.

Consider the following complex of free $B^e$-modules:
$$\cdots\stackrel{d}{\to}\displaystyle \bigoplus_{\substack{ p+q=n+1\\ p>0\ q>0}}  B^qSB^p\stackrel{d}{\to}\cdots\stackrel{d}{\to}
\ BSB^2\oplus B^2SB \ \stackrel{d}{\to}\   BSB \ \stackrel{d}{\to} S\to 0,$$
where the first differential is given by $d(b\otimes s\otimes b')=b.s.b'$. In greater degrees, the differential is the differential of the total complex of the double complex which has $ B^qSB^p$ at the spot $(q,p)$, with vertical and horizontal differentials $ B^qSB^p\to B^qSB^{p-1}$ and $ B^qSB^p \to B^{q-1}SB^{p}$ given respectively by
\vskip1mm
$b_1\otimes \cdots\otimes b_q\otimes s \otimes b'_1\otimes\cdots\otimes b'_p\mapsto \\ (-1)^{q+1}[b_1\otimes \cdots\otimes b_q\otimes s.b'_1\otimes\cdots\otimes b'_p+\\
\sum_{i=1}^{p-1} (-1)^i b_1\otimes \cdots\otimes b_q\otimes s \otimes b'_1\otimes\cdots \otimes b'_ib'_{i+1} \otimes\cdots\otimes b'_p]$
\vskip1mm
and
\vskip1mm
$b_1\otimes \cdots\otimes b_q\otimes s \otimes b'_1\otimes\cdots\otimes b'_p\mapsto \\
\sum_{i=1}^{q-1} (-1)^i b_1\otimes \cdots \otimes b_ib_{i+1} \otimes \cdots \otimes b_q\otimes s \otimes b'_1\otimes\cdots\otimes b'_p+\\
(-1)^q b_1\otimes \cdots\otimes b_q.s \otimes b'_1\otimes\cdots\otimes b'_p$

For $p=2$ we proceed as in \cite{CLMS Pacific Erratum}, while for $p>2$ we will use the hypothesis $\Tor_*^B(S, S^{\otimes_B n})=0$ for $*>0$ and for all $n$.  Let $F_\bullet \to S$ be the above projective resolution  of $S$. Tensoring it over $B$ with the left bar resolution
$$\cdots \to B B S\to B S\to S\to 0$$ of $S$ yields the double complex $\mathbf D$:
\vskip3mm

\begin{tikzcd}
  & \vdots \arrow[d]                  & \vdots \arrow[d]                    & \vdots \arrow[d]                    &                  \\
0 & BBS\otimes_BS \arrow[d] \arrow[l] & BBS\otimes_BF_0 \arrow[d] \arrow[l] & BBS\otimes_BF_1 \arrow[d] \arrow[l] & \cdots \arrow[l] \\
0 & BS\otimes_BS \arrow[d] \arrow[l]  & BS\otimes_BF_0 \arrow[d] \arrow[l]  & BS\otimes_BF_1 \arrow[d] \arrow[l]  & \cdots \arrow[l] \\
0 & S\otimes_BS \arrow[l] \arrow[d]   & S\otimes_BF_0 \arrow[l] \arrow[d]   & S\otimes_BF_1 \arrow[l] \arrow[d]   & \cdots \arrow[l] \\
  & 0                                 & 0                                   & 0                                   &
\end{tikzcd}

\vskip3mm
The left bar resolution of $S$ has a right $B$-module contracting homotopy. Hence the columns are acyclic and the total complex is exact.

The bimodules are projective, except those of the bottom row and the left column. To obtain the required projective resolution of $S\otimes_B S$ as a $B$-bimodule, we proceed as in \cite{CIBILS2000}. Let $\mathbf X$ be the double subcomplex of  $\mathbf D$ given by the bottom row and the left column. We assert that $\mathsf{Tot}(\mathbf{D/X}) \to S\otimes_BS$ is a projective resolution of the $B$-bimodule $S\otimes_BS$. The homology of the bottom row is $\Tor_*^B(S,S)$, which is zero in positive degrees by hypothesis, this shows that the homology of $\mathsf{Tot}(\mathbf X)$ is zero in positive degrees. In degree zero its homology is $S\otimes_B S$, by a direct computation. Indeed, for surjective morphisms $f: Y\to X$ and $g:Z\to X$, and $(f,g): Y\oplus Z\to X$ we have that
${\ker (f,g)}/(\ker f\oplus\ker g)$ is isomorphic to  $X$
by the map induced by $f$ or by $g$.

The exact sequence of complexes $0\to \mathbf X \to \mathbf D\to \mathbf {D/X}\to 0$ gives a long exact sequence which shows that  $\mathsf{Tot}(\mathbf{D/X})$ has homology $S\otimes_B S$ in its last term. Otherwise it is exact, and is the required resolution of  $S\otimes_B S$.

Tensoring over $B$ the last resolution with the left bar resolution of $S$ and using the hypothesis $\Tor_*^B(S,S\otimes_B S)=0$ for $*>0$ yields a projective resolution of the $B$-bimodule $S^{\otimes_B 3}$. Continuing this way provides a  projective resolution of the $B$-bimodule $S^{\otimes_B p}$ for $p\geq 1$, and this shows that the homology of the $p$-th column is  $\Tor^{B^e}_{*}(X,S^{\otimes_Bp})$.

 Back to $G_p/G_{p-1}$, from the above we have $H_0(G'_p)= \Tor^{B^e}_0(X,S^{\otimes_Bp})$.
 The original complex  $G_p/G_{p-1}$ has $L_{p,0}$ in degree $0$, hence the image of the last differential of $G'_p$ is contained in $L_{p,0}\subset X\otimes S^{\otimes p}$.
 Recall that
 $$(X\otimes S^{\otimes p})/L_{p,0} = X\otimes_{B^e} S^{\otimes_B p} = \Tor^{B^e}_0(X,S^{\otimes_Bp}).$$
 The image of the last differential  of $G'_p$ is then $L_{p,0}$, that is the last differential of  $G_p/G_{p-1}$ is surjective and $H_0(G_p/G_{p-1})=0$.

\normalsize
 For $p=1$  it is proven in \cite[Proposition 3.3]{CLMSS} that there is a short exact sequence for Hochschild cohomology:
 $$0\to H^1(A|B,X)\stackrel{\iota}{\to} H^1(A,X) \stackrel{\kappa}{\to} H^1(B,X)$$

 Let $V'$ denote the dual of a vector space $V$. It is well known that for finite dimensional $A$ and $X$, we have $H_*(A,X)=(H^*(A,X'))' $. The same holds in the relative setting, and the result follows.
 \qed

\section{\sf  Jacobi-Zariski long exact sequences}\label{last}

The next result is a specialisation of the Jacobi-Zariski nearly exact sequence of the previous section, in order to obtain the long exact sequence of A. Kaygun in \cite{KAYGUN,KAYGUNe}.

For the convenience of the reader we provide a proof of the following result.
\begin{lemm}\label{flatflat}
Let $\Lambda$ be any $k$-algebra, and let $P$ and $Q$ be flat $\Lambda$-bimodules.
Then the $\Lambda$-bimodule $P\otimes_\Lambda Q$ is flat.
\end{lemm}
\begin{proof}
First we record that if $P$ is a flat bimodule, then it is both left and right flat. Indeed, let $X\hookrightarrow Y$ be an injection of right modules, and consider the inferred injection of bimodules
$\Lambda\otimes X\hookrightarrow\Lambda\otimes Y$. We have an injection
$$P\otimes_{\Lambda^e}(\Lambda\otimes X)\hookrightarrow P\otimes_{\Lambda^e}(\Lambda\otimes Y).$$
We have a natural isomorphism $P\otimes_{\Lambda^e}(\Lambda\otimes X) = X\otimes_\Lambda P$, sending $p\otimes\lambda\otimes x$ to $x\otimes p\lambda$. The result follows.

Let now $U\hookrightarrow V$ be an injection of right $\Lambda^e$-modules, we want to prove that $U\otimes_{\Lambda^e} (P\otimes_\Lambda Q)\rightarrow V\otimes_{\Lambda^e} (P\otimes_\Lambda Q)$ is an injection. Note that we have a natural isomorphism
$$U\otimes_{\Lambda^e} (P\otimes_\Lambda Q) = (U\otimes_{\Lambda} P)\otimes_{\Lambda^e} Q$$
induced by the identity of $U\otimes P\otimes Q$.
Now since $P$ is left flat, there is an injection of bimodules $U\otimes_{\Lambda} P \hookrightarrow V\otimes_{\Lambda} P$. Applying the exact functor $-\otimes_{\Lambda^e} Q$ gives the result.\qed
\end{proof}

We give now an alternative proof of results in \cite{KAYGUN,KAYGUNe}.
\begin{theo}
Let $B\subset A$ be an extension of $k$-algebras such that $A/B$ is a flat $B$-bimodule. Let $X$ be an $A$-bimodule. There is a Jacobi-Zariski long exact sequence
\begin{align*} \dots \stackrel{\delta}{\to} H_{m}(B,X) \stackrel{I}{\to} H_{m}(A,X) \stackrel{K}{\to} H_{m}(A|B,X) \stackrel{\delta}{\to}H_{m-1}(B,X) \stackrel{I}{\to}\dots \\\stackrel{\delta}{\to} H_{2}(B,X) \stackrel{I}{\to} H_{2}(A,X) \stackrel{K}{\to} H_{2}(A|B,X).
\end{align*}
If $A$ and $X$ are finite dimensional, then the Jacobi-Zariski long exact sequence ends at degree $1$:
\begin{align*} \dots \stackrel{\delta}{\to} H_{m}(B,X) \stackrel{I}{\to} H_{m}(A,X) \stackrel{K}{\to} H_{m}(A|B,X) \stackrel{\delta}{\to}H_{m-1}(B,X) \stackrel{I}{\to}\dots \\\stackrel{\delta}{\to} H_{1}(B,X) \stackrel{I}{\to} H_{1}(A,X) \stackrel{K}{\to} H_{1}(A|B,X).
\end{align*}
\end{theo}
\begin{proof}
 Since $A/B$ is flat as a $B$-bimodule, $A/B$ is flat as a left or right $B$-module (see for instance the first part of the proof of Lemma \ref{flatflat}). Then $\Tor_*^B(A/B, (A/B)^{\otimes_B n})=0$ for $*>0$ and for all $n$. By Theorem \ref{homology of JZ}, there is a spectral sequence converging to the gap of the long nearly exact sequence in degrees $\geq 2$. The first page of this spectral sequence is
$$E^1_{p,q}= \Tor^{B^e}_{q}(X,(A/B)^{\otimes_Bp}) \mbox{   \ \ for } p,q>0$$
and $0$ elsewhere.
By Lemma \ref{flatflat}, for $p>0$ we have that the $B^e$-module $(A/B)^{\otimes_B p}$ is flat. Hence, if $q>0$, we have $\Tor^{B^e}_{q}(X,(A/B)^{\otimes_Bp})=0$. Consequently the first page of the spectral sequence is $0$, so the graded vector space to which it converges is also $0$.

If $A$ and $X$ are finite dimensional, the second part of Theorem \ref{homology of JZ} provides the result.
\qed

\end{proof}

Next we improve the Jacobi-Zariski exact sequence ending at some degree obtained in \cite[Proposition 3.7]{CLMS Pacific} by not assuming that the extension splits.

\begin{defi}(\cite{CLMS Pacific})
Let $\Lambda$ be a $k$-algebra and let $M$ be a $\Lambda$-bimodule. The bimodule $M$ is \emph{tensor nilpotent} if there exists $n$ such that $M^{\otimes_\Lambda n}=0$. The index of nilpotency of $M$ is the smallest $n$ such that $M^{\otimes_\Lambda n}=0$.
\end{defi}

\begin{defi}(\cite{CLMS Pacific})
Let $\Lambda$ be a $k$-algebra and let $M$ be a $\Lambda$-bimodule. The bimodule $M$ is \emph{bounded} if it is tensor nilpotent, projective on one side, and of finite projective dimension as a bimodule.
\end{defi}

\begin{theo}
Let $B\subset A$ be an extension of $k$-algebras and let $X$ be an $A$-bimodule. Assume that $A/B$ is a bounded $B$-bimodule, with index of nilpotency $n$ and projective dimension $u$.

Then there is a Jacobi-Zariski long exact sequence

\begin{align*} \dots \stackrel{\delta}{\to} H_{m}(B,X) \stackrel{I}{\to} H_{m}(A,X) \stackrel{K}{\to} H_{m}(A|B,X) \stackrel{\delta}{\to}H_{m-1}(B,X) \stackrel{I}{\to}\dots \\\stackrel{\delta}{\to} H_{n(u+1)}(B,X) \stackrel{I}{\to} H_{n(u+1)}(A,X) \stackrel{K}{\to} H_{n(u+1)}(A|B,X).
\end{align*}

\end{theo}

\begin{proof}
For degrees $\geq 2$ the terms at page $1$ of the spectral sequence which converges to the gap are
$$E^1_{p,q}= \Tor^{B^e}_{q}(X,(A/B)^{\otimes_Bp}) \mbox{   \ \ for } p,q>0$$
and $0$ elsewhere.

Since $A/B$ is projective on one side and is of projective dimension $u$ as a $B$-bimodule, we have that $(A/B)^{\otimes_B p}$ is of projective dimension at most $pu$ as a $B$-bimodule (see \cite[Chapter IX, Proposition 2.6]{CARTANEILENBERG}).

If $p\geq n$ or $q>pu$, then $E_{p,q}^1=0$. Hence if $p+q\geq n(u+1)$, then $E_{p,q}^1=0$ and the gap is $0$ in degrees at least $n(u+1)$.\qed
\end{proof}

\footnotesize
\noindent C.C.:\\
Institut Montpelli\'{e}rain Alexander Grothendieck, CNRS, Univ. Montpellier, France.\\
{\tt Claude.Cibils@umontpellier.fr}

\medskip
\noindent M.L.:\\
Instituto de Matem\'atica y Estad\'\i stica  ``Rafael Laguardia'', Facultad de Ingenier\'\i a, Universidad de la Rep\'ublica, Uruguay.\\
{\tt marclan@fing.edu.uy}

\medskip
\noindent E.N.M.:\\
Departamento de Matem\'atica, IME-USP, Universidade de S\~ao Paulo, Brazil.\\
{\tt enmarcos@ime.usp.br}

\medskip
\noindent A.S.:
\\IMAS-CONICET y Departamento de Matem\'atica,
 Facultad de Ciencias Exactas y Naturales,\\
 Universidad de Buenos Aires, Argentina. \\{\tt asolotar@dm.uba.ar}

\end{document}